\documentclass[11pt]{amsart}
\usepackage[top=30truemm,bottom=30truemm,left=30truemm,right=30truemm]{geometry} 
\usepackage{amsmath, amssymb, array}
\usepackage[all]{xy}
\usepackage{ascmac}
\usepackage[dvips]{graphicx}
\usepackage{color}
\usepackage{amscd}
\usepackage{fancyhdr}
\usepackage[driverfallback=dvipdfm]{hyperref}
\usepackage{amsrefs}
\usepackage{cleveref}
\usepackage{listings}

\makeatletter
  \def\section{\@startsection{section}{1}{\z@}%
     {4ex plus 0ex}%
     {4ex plus 0ex}%
     {\normalfont\large\bfseries}}
  \makeatother

\makeatletter
  \def\subsection{\@startsection{subsection}{1}{\z@}%
     {2ex plus 0ex}%
     {2ex plus 0ex}%
     {\normalfont\normalsize\bfseries}}
  \makeatother

\theoremstyle{plain}
 \newtheorem{theorem}{Theorem}[section]
 \crefname{theorem}{Theorem}{Theorems}
 
 \crefname{proposition}{Proposition}{Propositions}
 
 \crefname{lemma}{Lemma}{Lemmas}
 
 \crefname{corollary}{Corollary}{Corollaries}
 
 \crefname{conjecture}{Conjecture}{Conjectures}

 \crefname{question}{Question}{Questions}
 
 \crefname{problem}{Problem}{Problems}

\theoremstyle{definition} 
 
 \crefname{definition}{Definition}{Definitions}
 
 \crefname{example}{Example}{Examples}
 \newtheorem{remark}[theorem]{Remark}
 \crefname{remark}{Remark}{Remarks}

     \title{Counterexamples to the local-global principle associated with Swinnerton-Dyer's cubic form}

     \author{Yoshinosuke Hirakawa}
     \address[Yoshinosuke Hirakawa]{Department of Science and Technology, Keio University, 14-1, Hiyoshi 3-chome, Kohoku-ku, Yokohama-shi, Kanagawa-ken, Japan}
     \email{hirakawa@keio.jp}
     \thanks{This research was supported by JSPS KAKENHI Grant Number JP15J05818,
     the Research Grant of Keio Leading-edge Laboratory of Science \& Technology (Grant Numbers 2018-2019 000036 and 2019-2020 000074).
     This research was supported in part by KAKENHI 18H05233.
     This research was conducted as part of the KiPAS program FY2014--2018 of the Faculty of Science and Technology at Keio University as well as the JSPS Core-to-Core program ``Foundation of a Global Research Cooperative Center in Mathematics focused on Number Theory and Geometry".}

     \date{\today}


\begin{document}
     \begin{abstract}
     In this paper,
we imitate a classical construction of a counterexample
to the local-global principle of cubic forms of 4 variables
which was discovered first by Swinnerton-Dyer (Mathematica (1962)).
Our construction gives new explicit families of counterexamples
in homogeneous forms of $4, 5, 6, ..., 2n+2$ variables of degree $2n+1$ for infinitely many integers $n$.
It is contrastive to Swinnerton-Dyer's original construction that
we do not need any concrete calculation in the proof of local solubility.
     \end{abstract}
     \maketitle

%%%Contents

%%%%%
%%%%% intro
%%%%%
\section{Introduction}

In number theory,
it is one of the most classical subjects
to determine the set of solutions of a given algebraic equation
in the field $\mathbb{Q}$ of rational numbers.
In particular,
it is important as the first step to determine whether
a given equation has a (non-trivial) rational solution or not.

The first non-trivial objects are quadratic polynomials,
and this case culminates with a classical theorem of Minkowski (cf. \cite[Theorem 8, Ch. IV]{Serre})
which (essentially) states that
every quadratic homogeneous polynomial (i.e., quadratic form) with coefficients in $\mathbb{Q}$
has a non-trivial root in $\mathbb{Q}$
if and only if it has a non-trivial root in the field $\mathbb{R}$ of real numbers
and the field $\mathbb{Q}_{p}$ of $p$-adic numbers for every prime number $p$.
This is the so called local-global principle of quadratic forms.
\footnote{
\cite{Aitken-Lemmermeyer} is a well-written expository article on the local-global principle.
}

Contrary to the above great success in quadratic case,
many counterexamples to the local-global principle have been discovered in higher degrees.
For example, Selmer \cite{Selmer} constructed the famous counterexample
\[
	3x^{3}+4y^{3}+5z^{3},
\]
and Swinnerton-Dyer \cite{Swinnerton-Dyer} constructed the counterexample
\[
	t(t+x)(2t+x) - \prod_{m = 1}^{3}(x+\theta_{m}y+\theta_{m}^{2}z)
	 \quad (\theta_{m} := 2-2\cos(2\pi m/7)).
\]
For other counterexamples, see e.g.
\cite{Aitken-Lemmermeyer,Bremner,Browning-Heath-Brown,Cassels-Guy,Fujiwara-Sudo,Mordell,Jahnel}.
Among them, Mordell \cite{Mordell} and Jahnel \cite{Jahnel} generalized the above construction by Swinnerton-Dyer to infinite families of cubic forms of 4 variables.

In this paper, we imitate the above construction by Swinnerton-Dyer
to obtain new explicit infinite families of counterexamples to the local-global principle.
It is contrastive to Swinnerton-Dyer's original construction that
we do not need any concrete calculation in the proof of local solubility.
Indeed, we shall deduce it
from the Hasse-Weil lower bound to the number of $\mathbb{F}_{p}$-rational points
on non-singular projective curves over $\mathbb{F}_{p}$.
Moreover, unlike \cite{Mordell} nor \cite{Jahnel},
our construction gives explicit counterexamples of different dimensions at one time.
Note also that any one of our counterexamples (even restricted to $N = 7$)
is contained neither in \cite{Mordell} nor \cite{Jahnel}.

\begin{theorem} \label{main}
Let $n$ be a positive integer such that $N = 4n+3$ is a prime number.
Let $\alpha_{0}, \alpha_{1}, ..., \alpha_{n-1}, \beta, \gamma$ be integers
such that $\beta \geq 1$ and $2 \leq \gamma \leq 2n$.
Set $\theta_{m} = 2-2\cos(2 \pi m/N)$ ($m = 1, 2, ..., 2n+1$).
Assume that
\begin{enumerate}
\item
$\alpha_{0}(\alpha_{0}+1)$ is divisible by every prime number
smaller than $4n^{2}(2n-1)^{2}$ except for $N$, \\

\item
$\alpha_{0}(\alpha_{0}+1) \equiv \pm1 \pmod{N}$, and \\

\item
$\gcd(\alpha_{0}(\alpha_{0}+1), \alpha_{1}) = 1$ if $n \geq 2$.
\end{enumerate}
Then, the following polynomial is a counterexample to the local-global principle
of homogeneous forms of $\gamma+2$ variables of degree $2n+1$:
\[
	t(\alpha_{0}t^{n} + \sum_{i = 1}^{n-1}\alpha_{i}Nt^{n-i}x^{i}+N^{\beta}x^{n})
		((\alpha_{0}+1)t^{n}+\sum_{i = 1}^{n-1}\alpha_{i}Nt^{n-i}x^{i}+N^{\beta}x^{n})
			- \prod_{m = 1}^{2n+1}(x+\sum_{i = 1}^{\gamma} \theta_{m}^{i}y_{i}).
\]
\end{theorem}

\begin{remark}
We can find infinitely many pairs $(N, \alpha_{0})$
for which the whole conditions in Theorem \ref{main} hold:
Indeed, since the second condition is equivalent to
%"N : odd" is implicitly used here
$(2\alpha_{0}+1)^{2} \equiv 5, -3 \pmod{N}$,
the quadratic reciprocity law and Dirichlet's theorem on arithmetic progressions
ensure that we can find infinitely many $\alpha_{0}$
for which the second condition holds
whenever $N \equiv \pm1 \pmod{5}$ or $N \equiv 1 \pmod{3}$.
Therefore,
the Chinese remainder theorem ensures that,
for every prime number $N$
satisfying $N \equiv 3 \pmod{4}$ and $N \not\equiv 2, 8 \pmod{15}$,
%N ≠ 3, 5, 6, 9, 10, 12, 15 mod 15に注意.
we can find infinitely many $\alpha_{0}$ (with infinitely many $\alpha_{1}$ and other parameters)
for which the whole conditions in Theorem \ref{main} hold.
\end{remark}

\begin{remark}
It is plausible that our counterexamples may be explained by
the Brauer-Manin obstruction (cf. \cite{Jahnel}).
However, since our hypersurfaces
have singularities in general,
such a geometric argument should involve more technical calculations.
One of advantages of our proof is its almost purely-algebraic character,
which depends on
the cyclotomic field theory,
the Hasse-Weil bound,
and Hensel's lemma.
\end{remark}

This paper is organized as follows: In the next section, we prove that our polynomials have no non-trivial roots in $\mathbb{Q}$ in a more generalized form.
The core of our proof of the global unsolubility is exactly the same as Swinnerton-Dyer's original proof (cf. \S2)
although our counterexamples do not contain his counterexample itself.
\footnote{
Strictly speaking,
our construction gives a generalization of
the following cubic form,
which looks like Swinnerton-Dyer's cubic form:
\[
	t(2t+x)(3t+x) - \prod_{m = 1}^{3}(x+\theta_{m}y+\theta_{m}^{2}z)
	 \quad (\theta_{m} := 2-2\cos(2\pi m/7)).
\]
}
In the third section, we prove that our polynomials have non-trivial roots in $\mathbb{Q}_{p}$ for every prime number $p$ in a more generalized form.
This is sufficient to ensure the local solubility
because our polynomials have odd degrees e.g. with respect to $t$,
and so have non-trivial roots in $\mathbb{R}$.

%%%%%
%%%%% grobal
%%%%%
\section{Global unsolubility}

In this section, we prove the following theorem, which implies that our polynomials in Theorem \ref{main} have no non-trivial roots in $\mathbb{Q}$.

\begin{theorem} \label{global_1}
Let $n$ be a positive integer such that $N = 4n+3$ is a prime number.
Let $\alpha_{0}, \tilde{\alpha}_{1}, \tilde{\alpha}_{2}, ..., \tilde{\alpha}_{n-1}, \beta$ be integers
such that $\beta \geq 0$.
Set $\theta_{m} = 2-2\cos(2\pi m/N)$ ($m = 1, 2, ..., 2n+1$).
Assume that $\alpha_{0}(\alpha_{0}+1)$ is prime to $N$.
Then, the polynomial
\[
	t(\alpha_{0}t^{n} + \sum_{i = 1}^{n-1} \tilde{\alpha}_{i} t^{n-i}x^{i}+N^{\beta}x^{n})
		((\alpha_{0}+1)t^{n}+\sum_{i = 1}^{n-1} \tilde{\alpha}_{i} t^{n-i}x^{i}+N^{\beta}x^{n})
			-\prod_{m = 1}^{2n+1} (x+\sum_{i = 1}^{2n} \theta_{m}^{i}y_{i})
\]
has no non-trivial roots in $\mathbb{Q}$.
\end{theorem}

\begin{proof}
We prove our assertion by contradiction.
Suppose that $(t, x, y_{1}, ..., y_{2n}) \neq (0, ..., 0)$ is a non-trivial root of the above polynomial.

First, note that $t \neq 0$.
Indeed, since $[\mathbb{Q}(\theta_{m}) : \mathbb{Q}] = 2n+1$,
$t = 0$ implies that $x = y_{1} = \cdots = y_{2n} = 0$, which is a contradiction.
Therefore, we may assume that $t$ and $x$ are coprime integers
by multiplying $t, x, y_{1}, ..., y_{2n}$ by a same suitable rational number.
Here, note that $y_{1}, ..., y_{2n}$ are not necessarily integers.

Next, we prove that
$\prod_{m = 1}^{2n+1} (x+\sum_{i = 1}^{2n} \theta_{m}^{i}y_{i})$
is prime to $N$ by contradiction.
Suppose that it is divisible by $N$.
Then, since the above product is a norm of $x+\sum_{i = 1}^{2n} \theta_{1}^{i}y_{i}$
with respect to $\mathbb{Q}(\theta_{1})/\mathbb{Q}$,
and $\theta_{1}$ generates the unique prime ideal over $N$,
we see that $x+\sum_{i = 1}^{2n} \theta_{m}^{i}y_{i}$ is divisible by $\theta_{1}$ for each $m$.
Since the $\theta_{1}$-adic valuations of $x, \theta_{m}y_{1}, ..., \theta_{m}^{2n}y_{2n}$ are distinct,
we see that all of $y_{1}, ..., y_{2n}$ are $\theta_{1}$-adic integers.
Hence, $x$ is divisible by $\theta_{1}$, so by $N$.
However, since we assume that $\alpha_{0}(\alpha_{0}+1)$ is prime to $N$,
we see that $t$ is also divisible by $N$, which is a contradiction.
Therefore, $\prod_{m = 1}^{2n+1} (x+\sum_{i = 1}^{2n} \theta_{m}^{i}y_{i})$ is prime to $N$.

As a consequence, we see that
$t$,
$\alpha_{0}t^{n} + \sum_{i = 0}^{n-1} \tilde{\alpha}_{i} t^{n-i}x^{i}+N^{\beta}x^{n}$, and
$(\alpha_{0}+1)t^{n}+\sum_{i = 1}^{n-1} \tilde{\alpha}_{i} t^{n-i}x^{i}+N^{\beta}x^{n}$
are pairwise coprime integers.
On the other hand, their product 
is a norm of an algebraic integer $x+\sum_{i = 1}^{2n} \theta_{m}^{i}y_{i}$,
each of themselves must be the norm of an integral ideal of $\mathbb{Q}(\theta_{1})$.
Suppose that $t$ has a prime divisor $q$ whose residual degree in $\mathbb{Q}(\theta_{1})/\mathbb{Q}$ is $e$.
Then, the $q$-adic valuation of $t$ must be divisible by $e$.
Moreover, the residual degree of $q$ in $\mathbb{Q}(\exp(2\pi i/N))/\mathbb{Q}$
coincides with the order of $q \pmod{N}$ in $(\mathbb{Z}/N\mathbb{Z})^{\times}$ (cf. \cite{Birch65}).
Therefore, we must have $t \equiv \pm 1 \pmod{N}$,
and similarly for the other two factors.
However, this implies that
\begin{align*}
	(\pm 1)^{n}
	&\equiv t^{n}
	= \left( (\alpha_{0}+1)t^{n}+\sum_{i = 1}^{n-1} \tilde{\alpha}_{i} t^{n-i}x^{i}+N^{\beta}x^{n} \right)
		- \left( \alpha_{0}t^{n} + \sum_{i = 0}^{n-1} \tilde{\alpha}_{i} t^{n-i}x^{i}+N^{\beta}x^{n} \right) \\
	&\equiv (\pm1) - (\pm1) \pmod{N},
\end{align*}
which is a contradiction as desired.
\end{proof}

We may also prove the following theorems in similar manners:

\begin{theorem} \label{global_2}
Let $n \geq 2$ be an integer such that $N = 4n+1$ is a prime number.
Let $\alpha_{0}, \tilde{\alpha}_{1}, \tilde{\alpha}_{2}, ..., \tilde{\alpha}_{n-2}, \beta$
be integers such that $\beta \geq 0$.
Set $\theta_{m} = 2-2\cos(2\pi m/N)$ ($m = 1, 2, ..., 2n$).
Assume that $\alpha_{0}(\alpha_{0}+1)$ is prime to $N$.
Then,
\[
	t^{2}(\alpha_{0}t^{n-1} + \sum_{i = 1}^{n-2} \tilde{\alpha}_{i} t^{n-1-i}x^{i}+N^{\beta}x^{n-1})
		((\alpha_{0}+1)t^{n-1}+\sum_{i = 1}^{n-2} \tilde{\alpha}_{i} t^{n-1-i}x^{i}+N^{\beta}x^{n-1})
			- \prod_{m = 1}^{2n} (x+\sum_{i = 1}^{2n-1} \theta_{m}^{i}y_{i})
\]
has no non-trivial roots in $\mathbb{Q}$.
\footnote{
Indeed,
since we assume that $N > 5$,
we can deduce a contradiction from
$t^{n-1} \equiv (\pm1)-(\pm1) \pmod{N}$ and $(t^{2})^{n-1} \equiv \pm 1 \pmod{N}$.
On the other hand, if $N = 5$,
our polynomials become quadratic polynomials,
hence the global solubility is reduced to the local-solubility via Minkowski's theorem.
}
\end{theorem}

\begin{theorem} \label{global_3}
Let $n \geq 2$ be an integer such that $N = 2n+1$ is a prime number.
Let $\alpha_{0}, \tilde{\alpha}_{1}, \tilde{\alpha}_{2}, ..., \tilde{\alpha}_{n-2}, \beta$ be integers such that $\beta \geq 0$.
Set $\theta_{m} = 1-\exp(2\pi im/N)$ ($m = 1, 2, ..., 2n$).
Assume that $\alpha_{0}(\alpha_{0}+1)$ is prime to $N$.
Then, the polynomial
\[
	t^{2}(\alpha_{0}t^{n-1} + \sum_{i = 1}^{n-2} \tilde{\alpha}_{i} t^{n-1-i}x^{i} + N^{\beta}x^{n-1})
		((\alpha_{0}+1)t^{n-1}+\sum_{i = 1}^{n-2} \tilde{\alpha}_{i} t^{n-1-i}x^{i} + N^{\beta}x^{n-1})
			- \prod_{m = 1}^{2n} (x+\sum_{i = 1}^{2n-1} \theta_{m}^{i}y_{i})
\]
has no non-trivial roots in $\mathbb{Q}$.
\footnote{Indeed, we can deduce a contradiction from $t^{n-1} \equiv 1-1 \pmod{N}$ and $(t^{2})^{n-1} \equiv 1 \pmod{N}$.}
\end{theorem}

\begin{remark}
The Chevalley-Warning theorem (cf. \cite{Chevalley}, \cite[Theorem 3, Ch. I]{Serre}) ensures that
the polynomials in Theorems \ref{global_1}, \ref{global_2}, and \ref{global_3} modulo $p$ has a non-trivial root in $\mathbb{F}_{p}$ for every prime number $p$.
\end{remark}

%%%%%
%%%%% local
%%%%%
\section{Local solubility}

In this section, we prove the following theorem,
which implies that our polynomials in Theorem \ref{main} have non-trivial roots in $\mathbb{Q}_{p}$
for every prime number $p$.

\begin{theorem} \label{local}
Let $p$ be a prime number.
Let $n$ be a positive integer such that $N = 4n+3$ is a prime number.
Let $\alpha_{0}, \alpha_{1}, ..., \alpha_{n}, \beta_{1}, ..., \beta_{n}$ be integers.
Set $\theta_{m} = 2-2\cos(2\pi m/N)$ ($m = 1, 2, ..., 2n+1$) and 
\begin{align*}
	f(t, x, y_{1}, y_{2})
	&= t(\alpha_{0} t^{n}+\sum_{i = 1}^{n} \alpha_{i}N t^{n-i}x^{i})
		((\alpha_{0}+1)t^{n}+\sum_{i = 1}^{n} \beta_{i}Nt^{n-i}x^{i}) \\
	& - \prod_{m = 1}^{2n+1}(x+\theta_{m}y_{1}+\theta_{m}^{2}y_{2}).
\end{align*}
Assume that
\begin{enumerate}
\item
$\alpha_{0}(\alpha_{0}+1)$ is divisible by every prime number smaller than $4n^{2}(2n-1)^{2}$ except for $N$, \\

\item
$\alpha_{0}(\alpha_{0}+1) \equiv \pm1 \pmod{N}$, and \\

\item
$\gcd(\alpha_{0}, \alpha_{1}) = \gcd(\alpha_{0}+1, \beta_{1}) = 1$.
\end{enumerate}
Then, the projective surface $S$ defined by $f(t, x, y_{1}, y_{2})$ has a $\mathbb{Q}_{p}$-rational point.
\end{theorem}

\begin{proof}
First, note that $S$ has a $\mathbb{Q}_{p}$-rational point
whenever $p \geq 4n^{2}(2n-1)^{2}$ ($> N$) and $p$ is prime to $\alpha_{0}(\alpha_{0}+1)$:
Indeed, since in this case the polynomial
\[
	f(t, 0, y_{1}, y_{2})
	= \alpha_{0}(\alpha_{0}+1)t^{2n+1}-N\prod_{m = 1}^{2n+1}(y_{1}+\theta_{m}y_{2})
\]
defines a non-singular curve of genus $n(2n-1)$ over $\mathbb{F}_{p}$,
the Hasse-Weil bound (cf. \cite{Bombieri}, \cite{Weil}) ensures that
it has an $\mathbb{F}_{p}$-rational point.
Therefore, Hensel's lemma (cf. \cite[Corollary 1, Ch. II]{Serre}) ensures that
$S$ has a $\mathbb{Q}_{p}$-rational point.
Thus, it is sufficient to consider the cases
where $p < 4n^{2}(2n-1)^{2}$ or $p$ divides $\alpha_{0}(\alpha_{0}+1)N$.
By the first assumption, the first case is included in the second case.

Suppose that $p$ ($\neq N$) divides $\alpha_{0}$.
Then, the polynomial
\[
	f(1, x, 0, 0)
	\equiv \alpha_{1}Nx+\cdots+\alpha_{n}\beta_{n}N^{2}x^{2n}-x^{2n+1} \pmod{p}
\]
has a root $x = 0$ in $\mathbb{F}_{p}$,
which is not a zero of the $x$-derivative
since we assume that $\alpha_{1}$ is prime to $\alpha_{0}$.
Therefore, Hensel's lemma ensures that $S$ has a $\mathbb{Q}_{p}$-rational point.
The case where $p$ divides $\alpha_{0}+1$ is similar.

Finally, suppose that $p = N$. Then, the polynomial
\[
	f(1, x, y_{1}, y_{2})
	\equiv \alpha_{0}(\alpha_{0}+1)-x^{2n+1} \pmod{N}
\]
has a root $x = \pm1$ in $\mathbb{F}_{N}$
since we assume that $\alpha_{0}(\alpha_{0}+1) \equiv \pm1 \pmod{N}$,
which is not a zero of the $x$-derivative.
Therefore, Hensel's lemma ensures that $S$ has a $\mathbb{Q}_{N}$-rational point.
This completes the proof.
\end{proof}

\begin{remark}
If we replace $N = 4n+3$ to $4n+1$ or $\theta_{m} = 2-2\cos(2\pi m/N)$ to $1-\exp(2\pi m/N)$
as in Theorems \ref{global_2} and \ref{global_3},
then we obtain a counterpart of Theorem \ref{local}.
The detail is left to interested readers.
\end{remark}

\section*{Acknowledgements}

The author would like to express his sincere gratitude to Yoshinori Kanamura and Yosuke Shimizu for discussing related topics. In particular, Kanamura's question on how to construct explicit infinite families of irreducible and higher-dimensional counterexamples to the Hasse principle was the starting point of this study. The author would like to express his sincere gratitude also to his advisor Prof. Ken-ichi Bannai, Yoshinori Kanamura, and  Masataka Ono for reading the draft of this manuscript and giving many valuable comments to make it easy to read.

\end{document}